\documentclass[letterpaper,10 pt,conference]{ieeeconf}
\usepackage[utf8]{inputenc}
\usepackage{subfiles}
\usepackage{bm}
\usepackage{bbm}
\usepackage{amsmath}
\usepackage{amsfonts}
\usepackage{amssymb}
\usepackage{mathtools}
\let\labelindent\relax \usepackage{enumitem}   
\usepackage{graphicx}   
\usepackage{color}
\usepackage{comment}
\usepackage{graphicx}
\usepackage{cite}
\usepackage{float}
\usepackage{stfloats}
\usepackage{pmat}
\usepackage{lipsum}
\usepackage[hidelinks]{hyperref}
\newtheorem{theorem}{Theorem}   
\newtheorem{lemma}{Lemma}

\newtheorem{proposition}{Proposition}
\newtheorem{corollary}{Corollary}
\newtheorem{example}{Example}
\newtheorem{remark}{Remark}

\newtheorem{definition}{Definition}
\newtheorem{assumption}{Assumption}

\DeclareMathOperator{\im}{im}

\DeclareMathOperator{\rank}{rank}

\DeclareMathOperator{\inte}{int}

\DeclareMathOperator{\inertia}{In}

\let\leq\leqslant
\let\geq\geqslant

\newcommand{\bmat}{\begin{matrix}}
\newcommand{\emat}{\end{matrix}}
\newcommand{\bbm}{\begin{bmatrix}}
\newcommand{\ebm}{\end{bmatrix}}
\newcommand{\bbma}{\begin{bmatrix*}[r]}
\newcommand{\ebma}{\end{bmatrix*}}
\newcommand{\bpm}{\begin{pmatrix}}
\newcommand{\epm}{\end{pmatrix}}
\newcommand{\bvm}{\begin{vmatrix}}
\newcommand{\evm}{\end{vmatrix}}
\newcommand{\bse}{\begin{subequations}}
\newcommand{\ese}{\end{subequations}}
\newcommand{\beq}{\begin{equation}}
\newcommand{\eeq}{\end{equation}}
\newcommand{\ben}{\renewcommand{\labelenumi}{\arabic{enumi}.}
\renewcommand{\theenumi}{\arabic{enumi}}\begin{enumerate}}

\newcommand{\een}{\end{enumerate}}

\newcommand{\beni}{\renewcommand{\labelenumi}{\roman{enumi}.}
\renewcommand{\theenumi}{\roman{enumi}}\begin{enumerate}}

\newcommand{\eeni}{\end{enumerate}}

\newcommand{\bena}{\renewcommand{\labelenumi}{\alph{enumi}.}
\renewcommand{\theenumi}{\alph{enumi}}\begin{enumerate}}

\newcommand{\eena}{\end{enumerate}}

\newcommand{\bit}{\begin{itemize}}
\newcommand{\eit}{\end{itemize}}
\newcommand{\bthe}{\begin{theorem}}
\newcommand{\ethe}{\end{theorem}}
\newcommand{\blem}{\begin{lemma}}
\newcommand{\elem}{\end{lemma}}
\newcommand{\bprop}{\begin{proposition}}
\newcommand{\eprop}{\end{proposition}}
\newcommand{\bex}{\begin{example}}
\newcommand{\eex}{\end{example}}
\newcommand{\bas}{\begin{assumption}}
\newcommand{\eas}{\end{assumption}}
\newcommand{\bre}{\begin{remark}}
\newcommand{\ere}{\end{remark}}
\newcommand{\bcor}{\begin{corollary}}
\newcommand{\ecor}{\end{corollary}}
\newcommand{\bdfn}{\begin{definition}}
\newcommand{\edfn}{\end{definition}}
\newcommand{\bcon}{\begin{conjecture}}
\newcommand{\econ}{\end{conjecture}}

\IEEEoverridecommandlockouts                              

\title{\LARGE \bf
Synthesis of Dissipative Systems Using Input-State Data
}

\author{Encho T. Nguyen
\thanks{Encho T. Nguyen is a student at the University of Groningen, The Netherlands. {\tt\small n.t.nguyen.6@student.rug.nl}} and Henk J. van Waarde
\thanks{Henk J. van Waarde is with the Bernoulli Institute for Mathematics, Computer Science, and Artificial Intelligence, University of Groningen, Nijenborgh~9, 9747 AG, Groningen, The Netherlands. {\tt\small h.j.van.waarde@rug.nl}}
}

\begin{document}

\maketitle
\thispagestyle{empty}
\pagestyle{empty}

\begin{abstract}
This paper deals with the data-driven synthesis of dissipative linear systems in discrete time. We collect finitely many noisy data samples with which we synthesise a controller that makes all systems that explain the data dissipative with respect to a given quadratic supply rate. By adopting the informativity approach, we introduce the notion of informativity for closed-loop dissipativity. Under certain assumptions on the noise and the system, with the help of tools for quadratic matrix inequalities, we provide necessary and sufficient conditions for informativity for closed-loop dissipativity. We also provide a recipe to design suitable controllers by means of data-based linear matrix inequalities. This main result comprises two parts, to account for both the cases that the output matrices are known or unknown. Lastly, we illustrate our findings with an example, for which we want to design a data-driven controller achieving (strict) passivity.
\end{abstract}

\section{Introduction}
Motivated by the increasing complexity of modern engineering systems and the abundance of available data, there has been a recent surge of interest in data-driven analysis and control. In many cases, the data do not give rise to a unique mathematical model of the system, which makes it appealing to avoid system identification and work with direct data-driven control methods instead \protect\cite{vanWaarde-2023-informativityapproach, DePersis-2020-datadrivencontrol, Coulson-2019-DeePC}. However, the success of any data-driven controller can only be guaranteed if the collected data are ``sufficiently rich". Because of this, the concept of \emph{data informativity} has been introduced (see \protect\cite{vanWaarde-2023-informativityapproach}), but conditions for informativity vary from problem to problem.

Over the years, the notion of \emph{dissipativity}, introduced by Jan C. Willems in \protect\cite{Willems-1972-dissipativitypart1} and \protect\cite{Willems-1972-dissipativitypart2}, has proven itself to be one of the most important concepts in systems and control, that is inseparable from modelling (physical) dynamical systems and controller synthesis. A system is called dissipative if the rate of change of the stored energy in the system does not exceed the supplied energy. Originally, the stored energy is expressed as a function of the state of the system, but, later on, Willems and Trentelman expanded upon the input-state-output framework by developing the behavioural approach with quadratic differential forms laying its foundation \protect\cite{Willems-1998-QDFs}. As a prime example of its application, in the two-part paper \protect\cite{WillemsTrentelman-2002-dissipativesystemsQDFspart1, TrentelmanWillems-2002-dissipativesystemsQDFspart2} the $\mathcal{H}_{\infty}$ control problem was posed and solved in a behavioural context.

Dissipativity for a linear system with a quadratic supply rate can be verified by the dissipation inequality, which can be rewritten as a linear matrix inequality (LMI) involving the system model. In the case that the system matrices are unknown, a number of papers have focused on the problem of verifying dissipativity properties from data. We mention the contributions \protect\cite{Maupong-2017-datadrivendissipativity}, \protect\cite{Romer-2017-iodissipationinequalities, Romer-2019-iodatadrivendissipativity, Koch-2022-datadrivendissipativity} and \protect\cite{vanWaarde-2022-isodatadrivendissipativity} that have tackled this problem in various scenarios involving input-state-output and input-output data that are either exact or noisy.

While \protect\cite{Maupong-2017-datadrivendissipativity, Romer-2017-iodissipationinequalities, Romer-2019-iodatadrivendissipativity, Koch-2022-datadrivendissipativity, vanWaarde-2022-isodatadrivendissipativity} focus on the analysis of dissipativity properties, in this paper, we are interested in designing controllers that achieve dissipative closed-loop behaviour. We have access to a batch of noisy input-state measurements obtained from the true, data-generating system, where we assume that the noise satisfies a quadratic matrix inequality (QMI). This leads to a set of dynamical systems that are consistent with the data. Our goal is to design a single controller that makes all consistent closed-loop systems dissipative with respect to a given quadratic supply rate. If this is possible, we call the data \emph{informative for closed-loop dissipativity}. We work with general quadratic supply rates that satisfy a certain inertia assumption. The strength of this approach lies in the fact that appropriate choices of the supply rate lead to different relevant control problems such as data-driven feedback passivation and $\mathcal{H}_{\infty}$ control, the latter of which has already received attention in \protect\cite{vanWaarde-2023-informativityapproach, Berberich-2020-datadrivenrobustcontrol} and \protect\cite{Steentjes-2022-iocrosscovariancebounds}.

In particular, the contributions of this paper are as follows:
\begin{enumerate}
    \item We define the concept of \emph{informativity for closed-loop dissipativity} for data obtained from an input-state-output system (Definitions \ref{informativitydefinitionknown} and \ref{informativitydefinitionunknown}).
    \item In Theorem \ref{maintheorem1}, we obtain necessary and sufficient LMI conditions under which noisy data are informative for closed-loop dissipativity.
    The theorem allows the use of prior knowledge on how the noise affects the dynamics.
    Furthermore, we provide an explicit formula for a controller. We consider two cases -- the output and feedthrough matrices are either unknown or known. In the case that they are known, we make use of an additional projection result, namely, Proposition \ref{PiW}.
\end{enumerate}

The outline of this paper is as follows. In Section \ref{section2} we recapitulate some results regarding QMIs and dissipativity. In Section \ref{section3} we formulate the problem. Afterwards, with the help of the preliminary results from Section \ref{section4}, we formulate and prove our main result in Section \ref{section5}. In Section \ref{section6} we consider an example of data-driven feedback passivation. Finally, the conclusions are provided in Section \ref{section7}.

\subsection{Notation}
The \emph{Moore-Penrose pseudo-inverse} of a real matrix $A$ is denoted by $A^{\dagger}$. The \emph{set of real symmetric $n \times n$ matrices} is denoted by $\mathbb{S}^n$. The \emph{inertia} of a symmetric matrix $A$ is denoted by $\inertia (A) = (\rho_{-}, \rho_{0}, \rho_{+})$, where $\rho_{-}$, $\rho_{0}$ and $\rho_{+}$ are the number of negative, zero and positive eigenvalues of $A$, respectively. Let $A \in \mathbb{S}^n$. If $x^TAx>0$, for all nonzero $x \in \mathbb{R}^n$, then $A$ is called \emph{positive definite}, denoted by $A>0$. If $x^TAx \geq 0$ for all $x \in \mathbb{R}^n$, then $A$ is called \emph{positive semidefinite}, denoted by $A \geq 0$. \emph{Negative definiteness} and \emph{negative semidefiniteness} are defined similarly and denoted by $A<0$ and $A \leq 0$, respectively. The inequality $A>B$ implies that $A-B>0$. In addition, $A \geq B$, $A<B$ and $A \leq B$ are defined similarly. The $n \times m$ zero matrix is denoted by $0_{n \times m}$ and the $n \times n$ identity matrix by $I_n$. The subscripts are omitted whenever the size is clear from the context. The \emph{interior} of a set $V$ is denoted by $\inte(V)$.

\section{Preliminaries} \label{section2}
In this section, we review some results on quadratic matrix inequalities that we use throughout the paper and we recap the concept of dissipativity. For more details and proofs, we refer to \protect\cite{vanWaarde-2022-isodatadrivendissipativity} and \protect\cite{vanWaarde-2023-QMIs}.
\subsection{Sets Induced by QMIs}
For reasons that become clear in the next section, we are interested in the set
\begin{equation*}
    \mathcal{Z}_r(\Pi) \coloneqq \left\{ Z \in \mathbb{R}^{r \times q} \mid \begin{bmatrix} I_q \\ Z \end{bmatrix}^T \Pi \begin{bmatrix} I_q \\ Z \end{bmatrix} \geq 0 \right\},
\end{equation*}
where $\Pi \in \mathbb{S}^{q+r}$ is partitioned as
\begin{equation*}
    \begin{bmatrix} \Pi_{11} & \Pi_{12} \\ \Pi_{21} & \Pi_{22} \end{bmatrix}
\end{equation*}
with $\Pi_{11} \in \mathbb{S}^{q}$ and $\Pi_{22} \in \mathbb{S}^{r}$. Throughout the paper, $\Pi \in \mathbb{S}^{q+r}$ implies this particular partitioning. The set $\mathcal{Z}_r(\Pi)$ is nonempty and convex if the following three conditions hold:
\begin{equation} \label{nonemptyconvexconditions}
    \Pi_{22} \leq 0, \ \Pi|\Pi_{22} \geq 0, \ \ker \Pi_{22} \subseteq \ker \Pi_{12},
\end{equation}
where $\Pi|\Pi_{22} \coloneqq \Pi_{11} - \Pi_{12}\Pi_{22}^{\dagger}\Pi_{21}$ denotes the (generalised) Schur complement of $\Pi$ with respect to $\Pi_{22}$. Define the set
\begin{equation*}
    \boldsymbol{\Pi}_{q,r} \coloneqq \left\{ \Pi \in \mathbb{S}^{q+r} \mid \text{(\ref{nonemptyconvexconditions})} \ \text{hold} \right\}.
\end{equation*}
Given that $\Pi \in \boldsymbol{\Pi}_{q,r}$, the set $\mathcal{Z}_r(\Pi)$ is bounded if and only if $\Pi_{22}<0$.

For $W \in \mathbb{R}^{q \times p}$, $\mathcal{S} \subseteq \mathbb{R}^{r \times q}$ and $\Pi \in \mathbb{S}^{q+r}$, we define $\mathcal{S}W \coloneqq \{SW \colon S \in \mathcal{S}\}$ and
\begin{equation*}
    \Pi_W \coloneqq \begin{bmatrix} W^T & 0 \\ 0 & I_r \end{bmatrix} \Pi \begin{bmatrix} W & 0 \\ 0 & I_r \end{bmatrix} \in \mathbb{S}^{p+r}.
\end{equation*}
If $\Pi \in \boldsymbol{\Pi}_{q,r}$, then $\Pi_W \in \boldsymbol{\Pi}_{p,r}$.
\begin{proposition} \label{PiW}
    Let $\Pi \in \boldsymbol{\Pi}_{q,r}$ and $W \in \mathbb{R}^{q \times p}$. If either $W$ has full column rank or $\Pi_{22}$ is nonsingular, then $\mathcal{Z}_r(\Pi)W = \mathcal{Z}_r(\Pi_W)$.
\end{proposition}

\subsection{Matrix Version of the S-lemma}
We recall necessary and sufficient conditions under which all solutions to one QMI also satisfy another QMI.
\begin{proposition} \label{dissmainproof1}
    Let $M, N \in \mathbb{S}^{q+r}$. Assume that $N \in \boldsymbol{\Pi}_{q,r}$ and $N$ has at least one positive eigenvalue. Then $\mathcal{Z}_r(N) \subseteq \mathcal{Z}_r(M)$ if and only if there exists a real $\alpha \geq 0$ such that $M-\alpha N \geq 0$.
\end{proposition}

\subsection{Dissipativity of Discrete-Time Linear Systems}
Consider a linear discrete-time input-state-output system
\begin{align} \label{noiselessdiss}
\begin{split}
    \textbf{x}(t+1)&=A\textbf{x}(t)+B\textbf{u}(t), \\
    \textbf{y}(t)&=C\textbf{x}(t)+D\textbf{u}(t),
\end{split}
\end{align}
where $A \in \mathbb{R}^{n \times n}$, $B \in \mathbb{R}^{n \times m}$, $C \in \mathbb{R}^{p \times n}$, $D \in \mathbb{R}^{p \times m}$ and $(\textbf{u}, \textbf{x}, \textbf{y}) \colon \mathbb{N} \to \mathbb{R}^{m + n + p}$.
\begin{definition}
    Let $S \in \mathbb{S}^{m+p}$. We call system (\ref{noiselessdiss}), or the quadruple $(A, B, C, D)$, \emph{dissipative with respect to the supply rate}
    \begin{equation} \label{supplyrate}
        s(u,y)= \begin{bmatrix} u \\ y \end{bmatrix}^T S \begin{bmatrix} u \\ y \end{bmatrix}
    \end{equation}
    if there exists a matrix $P \in \mathbb{S}^{n}$ such that $P \geq 0$ and the dissipation inequality
    \begin{equation} \label{dissineq}
        \textbf{x}(t)^T P \textbf{x}(t) + s(\textbf{u}(t),\textbf{y}(t)) \geq \textbf{x}(t+1)^T P \textbf{x}(t+1)
    \end{equation}
    holds for all $t \geq 0$ and all trajectories $(\textbf{u}, \textbf{x}, \textbf{y})$ of (\ref{noiselessdiss}).
\end{definition}

We can rewrite (\ref{dissineq}) as
\begin{equation} \label{disslmi}
    \begin{bmatrix} I & 0 \\ A & B \end{bmatrix}^T \begin{bmatrix} P & 0 \\ 0 & -P \end{bmatrix} \begin{bmatrix} I & 0 \\ A & B \end{bmatrix} + \begin{bmatrix} 0 & I \\ C & D \end{bmatrix}^T S \begin{bmatrix} 0 & I \\ C & D \end{bmatrix} \geq 0,
\end{equation}
so asking for system (\ref{noiselessdiss}) to be dissipative with respect to the supply rate (\ref{supplyrate}) is equivalent to requiring the feasibility of the LMIs $P \geq 0$ and (\ref{disslmi}).

\section{Problem Formulation} \label{section3}
Now consider the system
\begin{subequations} \label{noisediss}
    \begin{align}
        \textbf{x}(t+1)&=A_s\textbf{x}(t)+B_s\textbf{u}(t)+E\textbf{w}(t), \label{noisediss1}\\
        \textbf{y}(t)&=C_s\textbf{x}(t)+D_s\textbf{u}(t)+F\textbf{w}(t), \label{noisediss2}
    \end{align}
\end{subequations}
where $(A_s, B_s, C_s, D_s)$ denote the ``true'' system matrices. The state and input matrices $(A_s, B_s)$ and the noise term $\textbf{w} \in \mathbb{R}^d$ are unknown, whereas $(E, F) \in \mathbb{R}^{n \times d} \times \mathbb{R}^{p \times d}$ are known. The matrices $E$ and $F$ capture our prior knowledge on how the noise affects the dynamics of the system. If we lack such knowledge, then we can take them to be the identity matrix. We derive results for two scenarios. Namely, the cases that the matrices $C_s$ and $D_s$ are either known or unknown. The goal of this paper is to find a matrix $K \in \mathbb{R}^{m \times n}$ such that the static state feedback controller $\textbf{u}(t)=K\textbf{x}(t)$ makes the closed-loop system $(A_s\!+\!B_sK, E, C_s\!+\!D_sK, F)$ dissipative with respect to the supply rate
\begin{equation} \label{supplyratenoise}
    s(w,y) = \begin{bmatrix} w \\ y \end{bmatrix}^T S \begin{bmatrix} w \\ y \end{bmatrix},
\end{equation}
where $S \in \mathbb{S}^{d + p}$.
\begin{assumption} \label{Assumption: Inertia of S}
    The matrix $S$ has inertia $\inertia (S) = (p, 0, d)$.
\end{assumption}
Examples that satisfy this condition include the passive supply rate (for $d=p$) and the $\ell_2$-gain supply rate, where $S$ takes the form
\begin{equation*}
    \begin{bmatrix} 0 & I_d \\ I_d & 0 \end{bmatrix} \text{ and } \begin{bmatrix} \gamma^2 I_d & 0 \\ 0 & -I_p \end{bmatrix},
\end{equation*}
respectively, with $\gamma > 0$.

\begin{remark}
    Although both (\ref{noisediss1}) and (\ref{noisediss2}) use the same noise term $\textbf{w}$, system (\ref{noisediss}) is considered a generalisation of the case where there are different process and measurement noise terms. Indeed,
    \begin{align*}
        \textbf{x}(t+1)&=A_s\textbf{x}(t)+B_s\textbf{u}(t)+E\textbf{v}(t), \\
        \textbf{y}(t)&=C_s\textbf{x}(t)+D_s\textbf{u}(t)+F\textbf{z}(t)
    \end{align*}
    can be rewritten as
    \begin{align*}
        \textbf{x}(t+1)&=A_s\textbf{x}(t)+B_s\textbf{u}(t)+\begin{bmatrix} E & 0 \end{bmatrix} \begin{bmatrix} \textbf{v}(t) \\ \textbf{z}(t) \end{bmatrix}, \\
        \textbf{y}(t)&=C_s\textbf{x}(t)+D_s\textbf{u}(t)+\begin{bmatrix} 0 & F \end{bmatrix} \begin{bmatrix} \textbf{v}(t) \\ \textbf{z}(t) \end{bmatrix},
    \end{align*}
    which is again of the form (\ref{noisediss}) by taking $\textbf{w} = \begin{bmatrix} \textbf{v}^T & \textbf{z}^T \end{bmatrix}^T$.
\end{remark}

We collect finitely many input-state measurements of (\ref{noisediss}) in the matrices
\begin{align*}
    U_{-} &\coloneqq \begin{bmatrix} u(0) & u(1) & \cdots & u(T-1) \end{bmatrix}, \\
    X &\coloneqq \begin{bmatrix} x(0) & x(1) & \cdots & x(T) \end{bmatrix}.
\end{align*}
In addition, we define the following matrices:
\begin{align*}
    X_{-} &\coloneqq \begin{bmatrix} x(0) & x(1) & \cdots & x(T-1) \end{bmatrix}, \\
    X_{+} &\coloneqq \begin{bmatrix} x(1) & x(2) & \cdots & x(T) \end{bmatrix}.
\end{align*}
For the noise, we make the following assumption on $W_{-} \coloneqq \begin{bmatrix} w(0) & w(1) & \cdots & w(T-1) \end{bmatrix}$.
\begin{assumption}
    The noise matrix $W_{-}$ satisfies
\begin{equation} \label{noisemodeldiss}
    \begin{bmatrix} I \\ W_{-}^T \end{bmatrix}^T \Phi \begin{bmatrix} I \\ W_{-}^T \end{bmatrix} \geq 0,
\end{equation}
i.e. $W_{-}^T \in \mathcal{Z}_T(\Phi)$, with $\Phi \in \boldsymbol{\Pi}_{d,T}$ and $\Phi_{22} < 0$.
\end{assumption}
With an appropriate choice of $\Phi$, the noise model can capture various assumptions on the noise such as energy bounds, sample covariance bounds, etc. (see \protect\cite{vanWaarde-2023-QMIs}). Then, from (\ref{noisediss1}) we have that $X_{+} = A_sX_{-} + B_sU_{-} + EW_{-}$. We first consider the setting where $C_s$ and $D_s$ are known. We denote the set of systems consistent with the data, i.e. all $(A, B)$ satisfying
\begin{equation} \label{sigmaknown}
     X_{+} = AX_{-} + BU_{-} + EW_{-}
\end{equation}
for some $W_{-}$ satisfying (\ref{noisemodeldiss}), by
\begin{equation*}
     \Sigma_k \coloneqq \left\{ (A, B) \mid \text{(\ref{sigmaknown})} \; \text{holds for some} \; W_{-}^T \in \mathcal{Z}_T(\Phi) \right\}.
\end{equation*}
Clearly, $(A_s, B_s) \in \Sigma_k$. Because we cannot distinguish the true system from other systems in $\Sigma_k$, our controller should make all systems in $\Sigma_k$ dissipative with respect to (\ref{supplyratenoise}), inspiring the informativity approach.
\begin{definition}\label{informativitydefinitionknown}
    The data $(U_{-}, X)$ are called \emph{informative for closed-loop dissipativity} with respect to the supply rate (\ref{supplyratenoise}) if there exist matrices $K$ and $P \geq 0$ such that
    \begin{equation} \label{preknownverdissineq}
        \! \begin{bmatrix} I & 0 \\ \mathcal{A} & E \end{bmatrix}^T \begin{bmatrix} P & 0 \\ 0 & -P \end{bmatrix} \begin{bmatrix} I & 0 \\ \mathcal{A} & E \end{bmatrix} + \begin{bmatrix} 0 & I \\ \mathcal{C}_s & F \end{bmatrix}^T S \begin{bmatrix} 0 & I \\ \mathcal{C}_s & F \end{bmatrix} \geq 0
    \end{equation}
    for all $(A, B) \in \Sigma_k$, where $\mathcal{A} \coloneqq A + BK$ and $\mathcal{C}_s \coloneqq C_s + D_sK$.
\end{definition}

Next, we turn our attention to the case that $C_s$ and $D_s$ are unknown. We collect finitely many output measurements of (\ref{noisediss}) in the matrix
\begin{equation*}
    Y_{-} \coloneqq \begin{bmatrix} y(0) & y(1) & \cdots & y(T-1) \end{bmatrix}.
\end{equation*}
Consider the system of linear equations
\begin{equation} \label{sigmaunknown}
    \begin{bmatrix} X_{+} \\ Y_{-} \end{bmatrix} = \begin{bmatrix} A & B \\ C & D \end{bmatrix} \begin{bmatrix} X_{-} \\ U_{-} \end{bmatrix} + \begin{bmatrix} E \\ F \end{bmatrix}W_{-}
\end{equation}
in the unknown $(A, B, C, D)$. The set of systems explaining the data is now given by
\begin{equation*}
    \Sigma_u \coloneqq \left\{ (A, B, C, D) \mid \text{(\ref{sigmaunknown})} \; \text{holds for some} \; W_{-}^T \in \mathcal{Z}_T(\Phi) \right\} \!.
\end{equation*}
Obviously, $(A_s, B_s, C_s, D_s) \in \Sigma_u$.
\begin{definition} \label{informativitydefinitionunknown}
    The data $(U_{-}, X, Y_{-})$ are called \emph{informative for closed-loop dissipativity} with respect to the supply rate (\ref{supplyratenoise}) if there exist matrices $K$ and $P \geq 0$ such that
    \begin{equation} \label{preunknownverdissineq}
        \begin{bmatrix} I & 0 \\ \mathcal{A} & E \end{bmatrix}^T \begin{bmatrix} P & 0 \\ 0 & -P \end{bmatrix} \begin{bmatrix} I & 0 \\ \mathcal{A} & E \end{bmatrix} + \begin{bmatrix} 0 & I \\ \mathcal{C} & F \end{bmatrix}^T S \begin{bmatrix} 0 & I \\ \mathcal{C} & F \end{bmatrix} \geq 0
    \end{equation}
    for all $(A, B, C, D) \in \Sigma_u$, where $\mathcal{A} \coloneqq A + BK$ and $\mathcal{C} \coloneqq C + DK$. 
\end{definition}

In what follows, we require the following assumption on $\Sigma_k$ and $\Sigma_u$.
\begin{assumption}
    The sets $\Sigma_k$ and $\Sigma_u$ have nonempty interior.
\end{assumption}

\section{Preliminary Results} \label{section4}
In this section, we derive some results which inspire and help prove our main result.
\begin{lemma} \label{knownequivalenceconsistentset}
    We have that $(A, B) \in \Sigma_k$ if and only if $\begin{bmatrix} A & B \end{bmatrix}^T \in \mathcal{Z}_{n+m}(N_k)$, where
    \begin{equation*}
    N_k \coloneqq
    \begin{pmat}[{|}]
       I & X_{+} \cr \-
        0 & {\begin{matrix} -X_{-} \\ -U_{-} \end{matrix}} \cr
    \end{pmat}
    \begin{bmatrix} E & 0 \\ 0 & I \end{bmatrix} \Phi \begin{bmatrix} E^T & 0 \\ 0 & I \end{bmatrix}
    \begin{pmat}[{|}]
       I & X_{+} \cr \-
       0 & {\begin{matrix} -X_{-} \\ -U_{-} \end{matrix}} \cr
    \end{pmat}^T \! \! \!.
    \end{equation*}
\end{lemma}

\vspace{5pt}

\begin{proof}
    By definition, $(A, B) \in \Sigma_k$ if and only if (\ref{sigmaknown}) holds for some $W_{-}^T \in \mathcal{Z}_T(\Phi)$. From Proposition \ref{PiW}, we have that $\mathcal{Z}_T(\Phi)E^T = \mathcal{Z}_T(\Phi_{E^T})$, where
    \begin{equation*}
        \Phi_{E^T} \coloneqq \begin{bmatrix} E & 0 \\ 0 & I \end{bmatrix} \Phi \begin{bmatrix} E^T & 0 \\ 0 & I \end{bmatrix}.
    \end{equation*}
    Therefore, $(A, B) \in \Sigma_k$ if and only if $(X_+ - AX_- - BU_-)^T \in \mathcal{Z}_T(\Phi_{E^T})$, equivalently,
    \begin{equation} \label{noisemodelknownver}
        \begin{bmatrix} I \\ A^T \\ B^T \end{bmatrix}^T N_k \begin{bmatrix} I \\ A^T \\ B^T \end{bmatrix} \geq 0.
    \end{equation}
    The latter inequality holds if and only if $\begin{bmatrix} A & B \end{bmatrix}^T \in \mathcal{Z}_{n+m}(N_k)$.
\end{proof}

\begin{lemma} \label{unknownequivalenceconsistentset}
    The system $(A, B, C, D) \in \Sigma_u$ if and only if
    \begin{equation*}
        \begin{bmatrix} A & B \\ C & D \end{bmatrix}^T \in \mathcal{Z}_{n+m}(N_u),
    \end{equation*}
    where
    \begin{equation*}
    N_u \coloneqq
    \begin{pmat}[{|}]
       I & {\begin{matrix} X_{+} \\ Y_{-} \end{matrix}} \cr \-
       0 & {\begin{matrix} -X_{-} \\ -U_{-} \end{matrix}} \cr
    \end{pmat}
    \begin{pmat}[{|}] {\begin{matrix} E \\ F \end{matrix}} & 0 \cr \- 0 & I \cr \end{pmat} \Phi \begin{pmat}[{|}] {\begin{matrix} E \\ F \end{matrix}} & 0 \cr \- 0 & I \cr \end{pmat}^T
    \begin{pmat}[{|}]
       I & {\begin{matrix} X_{+} \\ Y_{-} \end{matrix}} \cr \-
       0 & {\begin{matrix} -X_{-} \\ -U_{-} \end{matrix}} \cr
    \end{pmat}^T \! \! \!.
    \end{equation*}
\end{lemma}

\vspace{5pt}

\begin{proof}
    We follow the same steps as in the proof of Lemma \ref{knownequivalenceconsistentset}. Define $G \coloneqq \begin{bmatrix} E^T & F^T \end{bmatrix}$. Here, we apply Proposition \ref{PiW} with $G$ instead of $E^T$, so $\mathcal{Z}_T(\Phi)G = \mathcal{Z}_T(\Phi_{G})$. Hence, $(A, B, C, D) \in \Sigma_u$ if and only if
    \begin{equation*}
        \begin{bmatrix} X_{+} \\ Y_{-} \end{bmatrix}^T - \begin{bmatrix} X_{-} \\ U_{-} \end{bmatrix}^T \begin{bmatrix} A & B \\ C & D \end{bmatrix}^T \in \mathcal{Z}_T(\Phi_{G}).
    \end{equation*}
    The latter is equivalent to
    \begin{equation} \label{noisemodelunknownver}
    \begin{pmat}[{}]
      I \cr \-
      {\begin{matrix} A^T & \! \! \! \! C^T \\ B^T & \! \! \! \! D^T \end{matrix}} \cr
    \end{pmat}^T
    N_u
    \begin{pmat}[{}]
      I \cr \-
      {\begin{matrix} A^T & \! \! \! \! C^T \\ B^T & \! \! \! \! D^T \end{matrix}} \cr
    \end{pmat} \geq 0
    \end{equation}
    or, in other words,
    \begin{equation*}
        \begin{bmatrix} A & B \\ C & D \end{bmatrix}^T \in \mathcal{Z}_{n+m}(N_u).
    \end{equation*}
    This proves the lemma.
\end{proof}

\begin{figure*}[b!]
\hrulefill
\begin{equation} \label{MN_k}
\hat{M}_k \coloneqq
\begin{pmat}[{||}]
    {\hat{R} + \hat{H}} & {\begin{matrix} 0 & 0 \\ -C_sQ & \! \! -C_sL^T \end{matrix}} & {\begin{matrix} 0 \\ D_sL \end{matrix}} \cr \-
    {\begin{matrix} 0 & \! \! -QC_s^T \\ 0 & \! \! -LC_s^T \end{matrix}} & 0 & {\begin{matrix} Q \\ L \end{matrix}} \cr \-
    {\begin{matrix} 0 & \! L^TD_s^T \end{matrix}} & {\begin{matrix} \; Q & \quad \quad L^T \end{matrix}} & Q \cr
\end{pmat} \! , \
\hat{N}_k \coloneqq
\begin{pmat}[{|}]
    I & X_{+} \cr
    0 & 0_{p \times T} \cr \-
    0 & {\begin{matrix} -X_{-} \\ -U_{-} \end{matrix}} \cr \-
    0 & 0_{n \times T} \cr
\end{pmat}
\Phi_{E^T}
\begin{pmat}[{|}]
    I & X_{+} \cr
    0 & 0_{p \times T} \cr \-
    0 & {\begin{matrix} -X_{-} \\ -U_{-} \end{matrix}} \cr \-
    0 & 0_{n \times T} \cr
\end{pmat}^T
\end{equation}
\end{figure*}

The quadratic matrix inequalities (\ref{noisemodelknownver}) and (\ref{noisemodelunknownver}) are in terms of the transposes of the matrices $A$, $B$, $C$ and $D$. This motivates the use of the following dualisation result (the proof can be found in \protect\cite[Proposition 4]{vanWaarde-2022-isodatadrivendissipativity}).
\begin{proposition} \label{dualisationineq}
    Consider a matrix $P \in \mathbb{S}^{n}$ with $P > 0$ and a matrix
    \begin{equation*}
        \begin{bmatrix} A & B \\ C & D \end{bmatrix} \in \mathbb{R}^{(n+p) \times (n+m)}.
    \end{equation*}
    Suppose that $S \in \mathbb{S}^{m+p}$ has inertia $\inertia (S) = (p, 0, m)$. Define $Q \coloneqq P^{-1}$ and
    \begin{equation*}
        \hat{S} \coloneqq \begin{bmatrix} 0 & -I_p \\ I_m & 0 \end{bmatrix} S^{-1} \begin{bmatrix} 0 & -I_m \\ I_p & 0 \end{bmatrix}.
    \end{equation*}
    Then,
    \begin{equation*}
        \! \begin{bmatrix} I & 0 \\ A & B \end{bmatrix}^T \begin{bmatrix} P & 0 \\ 0 & -P \end{bmatrix} \begin{bmatrix} I & 0 \\ A & B \end{bmatrix} + \begin{bmatrix} 0 & I \\ C & D \end{bmatrix}^T S \begin{bmatrix} 0 & I \\ C & D \end{bmatrix} \geq 0
    \end{equation*}
    if and only if
    \begin{equation*}
        \! \begin{bmatrix} I & \! \! \! \! 0 \\ A^T & \! \! \! \! C^T \end{bmatrix}^T \! \! \begin{bmatrix} Q & \! \! \! \! 0 \\ 0 & \! \! \! \! -Q \end{bmatrix} \! \! \begin{bmatrix} I & \! \! \! \! 0 \\ A^T & \! \! \! \! C^T \end{bmatrix} + \begin{bmatrix} 0 & \! \! \! \! \! I \\ B^T & \! \! \! \! \! D^T \end{bmatrix}^T \! \! \hat{S} \! \begin{bmatrix} 0 & \! \! \! \! \! I \\ B^T & \! \! \! \! \! D^T \end{bmatrix} \geq 0.
    \end{equation*}
\end{proposition}
\phantom{} \\
\noindent Note that the above dualisation result works under the assumption that $P > 0$. For this reason, we make use an auxiliary statement. Namely, an extension to Sylvester's law of inertia, which can be found in \protect\cite[Theorem 3.1]{Dancis-1986-sylvesterlawofinertia3}.
\begin{proposition} \label{Sylvester}
    Let $H \in \mathbb{S}^n$ and $M^THM$, where $M \in \mathbb{R}^{n \times m}$, have $\nu$ and $\hat{\nu}$ negative eigenvalues, respectively. Then, $\nu + (m - n) - \dim(\ker M) \leq \hat{\nu}$.
\end{proposition}
The following lemma allows us to later apply Proposition \ref{dualisationineq}.
\begin{lemma} \label{Lemma: P > 0}
    Let matrices $K$ and $P \geq 0$ be such that (\ref{preunknownverdissineq}) holds for all $(A, B, C, D) \in \Sigma_u$. Then, $P > 0$.
\end{lemma}
\begin{proof}
Take $\xi \in \ker P$. For arbitrary $\alpha \in \mathbb{R}$ and $\eta \in \mathbb{R}^d$, pre- and postmultiplying (\ref{preunknownverdissineq}) by $\begin{bmatrix} \alpha \xi^T & \eta^T \end{bmatrix}$ and its transpose yields
\begin{equation*}
    \begin{bmatrix} \alpha \\ \eta \end{bmatrix}^T \! \! \left( \! - \! \begin{bmatrix} A\xi & \! \! \! E \end{bmatrix}^T \! P \! \begin{bmatrix} A\xi & \! \! \! E \end{bmatrix} \! + \! \begin{bmatrix} 0 & \! \! \! \! I \\ \mathcal{C}\xi & \! \! \! \! F \end{bmatrix}^T \! \! S \! \begin{bmatrix} 0 & \! \! \! \! I \\ \mathcal{C}\xi & \! \! \! \! F \end{bmatrix} \! \right) \! \! \begin{bmatrix} \alpha \\ \eta \end{bmatrix} \geq 0
\end{equation*}
with $\mathcal{A}$ and $\mathcal{C}$ defined as in Definition \ref{informativitydefinitionunknown}. The positive semidefiniteness of $P$ implies that
\begin{equation*}
    \begin{bmatrix} 0 & \! \! \! I \\ \mathcal{C}\xi & \! \! \! F \end{bmatrix}^T \! \! S \! \begin{bmatrix} 0 & \! \! \! I \\ \mathcal{C}\xi & \! \! \! F \end{bmatrix} \geq 0
\end{equation*}
for all $(A, B, C, D) \in \Sigma_u$. From Proposition \ref{Sylvester} and Assumption \ref{Assumption: Inertia of S}, we see that
\begin{equation*}
    \dim \left( \ker \begin{bmatrix} 0 & \! \! \! I \\ \mathcal{C}\xi & \! \! \! F \end{bmatrix} \right) \geq 1
\end{equation*}
and, consequently, $\mathcal{C}\xi = 0$ for all $(A, B, C, D) \in \Sigma_u$. Let $(A, B, C, D) \in \inte(\Sigma_u)$ and define $C_\epsilon \coloneqq \epsilon ve_i^T$, where $\epsilon > 0$, $v \in \mathbb{R}^p$ is nonzero and $e_i \in \mathbb{R}^n$ is the \emph{i}th standard basis vector of $\mathbb{R}^n$. For sufficiently small $\epsilon$, $(A, B, C+C_\epsilon, D) \in \Sigma_u$. Hence, $(C + C_\epsilon + DK)\xi = C_\epsilon \xi = \epsilon ve_i^T\xi = 0$ for all $i = 1, \ldots, n$. Because $\epsilon$ and $v$ are nonzero, we must have that $e_i^T\xi = 0$ for all $i$. In other words, $\xi = 0$, which shows that $P$ is positive definite.
\end{proof}

\section{Main Result} \label{section5}
We are ready to state and prove the main theorem, which we divide into two parts. In part \ref{Thm1a} we treat the case where $C_s$ and $D_s$ are unknown. In part \ref{Thm1b} we assume that $C_s$ and $D_s$ are known, which requires an additional assumption on the data.
\begin{theorem} \label{maintheorem1}
For $Q \in \mathbb{S}^n$ and $L \in \mathbb{R}^{m \times n}$, define matrices $\hat{M}_k$ and $\hat{N}_k$ as in $(\ref{MN_k})$ as well as
\begin{align*}
\hat{M}_u \! \! &\coloneqq \! \! \!
\begin{pmat}[{||}]
    \hat{R} & 0 & 0 \cr \-
    0 & 0 & {\begin{matrix} Q \\ L \end{matrix}} \cr \-
    0 & {\begin{matrix} Q & \! \! \! L^T \end{matrix}} & Q \cr
\end{pmat} \! \!, \,
\hat{N}_u \! \! \coloneqq \! \! \!
\begin{pmat}[{|}]
    I & {\begin{matrix} X_{+} \\ Y_{-} \end{matrix}} \cr \-
    0 & {\begin{matrix} -X_{-} \\ -U_{-} \end{matrix}} \cr \-
    0 & 0_{n \times T} \cr
\end{pmat}
\! \Phi_G \!
\begin{pmat}[{|}]
    I & {\begin{matrix} X_{+} \\ Y_{-} \end{matrix}} \cr \-
    0 & {\begin{matrix} -X_{-} \\ -U_{-} \end{matrix}} \cr \-
    0 & 0_{n \times T} \cr
\end{pmat}^T \! \! \!,
\end{align*}
where
\begin{align*}
    G &\coloneqq \begin{bmatrix} E^T & \! \! F^T \end{bmatrix} \!, \;
    \hat{S} \coloneqq \begin{bmatrix} 0 & -I_p \\ I_d & 0 \end{bmatrix} S^{-1} \begin{bmatrix} 0 & -I_d \\ I_p & 0 \end{bmatrix}, \\
    \hat{R} &\coloneqq \begin{bmatrix} Q & 0 \\ 0 & 0 \end{bmatrix} + \begin{bmatrix} 0 & I \\ E^T & F^T \end{bmatrix}^T \hat{S} \begin{bmatrix} 0 & I \\ E^T & F^T \end{bmatrix}, \\
    \hat{H} &\coloneqq \begin{bmatrix} 0 & 0 \\ 0 & -C_sQC_s^T-C_sL^TD_s^T-D_sLC_s^T \end{bmatrix}.
\end{align*}
\begin{enumerate}[label=(\alph*)]
    \item \label{Thm1a} Assume that $\hat{N}_u$ has at least one positive eigenvalue. The data $(U_{-}, X, Y_{-})$, generated by (\ref{noisediss}) with noise model (\ref{noisemodeldiss}), are informative for closed-loop dissipativity with respect to the supply rate (\ref{supplyratenoise}) if and only if there exist a positive definite matrix $Q$, a matrix $L$ and a scalar $\alpha \geq 0$ such that $\hat{M}_u-\alpha \hat{N}_u \geq 0$.
    \item \label{Thm1b} Assume that $\rank \begin{bmatrix} X_{-}^T & U_{-}^T \end{bmatrix}^T = n+m$ and $\hat{N}_k$ has at least one positive eigenvalue. The data $(U_{-}, X)$, generated by (\ref{noisediss1}) with noise model (\ref{noisemodeldiss}), are informative for closed-loop dissipativity with respect to the supply rate (\ref{supplyratenoise}) if and only if exactly one of the following is true:
    \begin{enumerate}[label=(\roman*)]
        \item \label{thm1bP>0} there exist a positive definite matrix $Q$, a matrix $L$ and a scalar $\alpha \geq 0$ such that $\hat{M}_k-\alpha \hat{N}_k \geq 0$;
        \item \label{thm1bP=0} $\im C_s \subseteq \im D_s$ and
        \begin{equation*}
            \begin{bmatrix} I \\ F \end{bmatrix}^T S \begin{bmatrix} I \\ F \end{bmatrix} \geq 0.
        \end{equation*}
    \end{enumerate}
\end{enumerate}
If the conditions in \ref{Thm1a} or \ref{Thm1b}\ref{thm1bP>0} are satisfied, then a controller that achieves closed-loop dissipativity for all systems in $\Sigma_u$ and $\Sigma_k$, respectively, is $K = L Q^{-1}$. If, instead, the conditions in \ref{Thm1b}\ref{thm1bP=0} hold, then $K$ can be chosen such that $C_s + D_sK = 0$.
\end{theorem}

\begin{figure*}[b!]
\hrulefill
\begin{equation} \label{M_k}
M_k \coloneqq
\begin{pmat}[{|.}]
    {\hat{R} + H} & {\begin{matrix} 0 \\ -C_sQ-D_sKQ \end{matrix}} & {\begin{matrix} 0 \\ -C_sQK^T-D_sKQK^T \end{matrix}} \cr \-
    {\begin{matrix} 0 & -QC_s^T-QK^TD_s^T \\ 0 & -KQC_s^T-KQK^TD_s^T \end{matrix}} & {\begin{matrix} -Q \\ -KQ \end{matrix}} & {\begin{matrix} -QK^T \\ -KQK^T \end{matrix}} \cr
\end{pmat}
\end{equation}
\end{figure*}

\begin{proof}
We begin with the first statement. Suppose that the data $(U_{-}, X, Y_{-})$ are informative. Then, there exist matrices $K$ and $P \geq 0$ such that (\ref{preunknownverdissineq}) holds for all $(A, B, C, D) \in \Sigma_u$. We have from Lemma \ref{Lemma: P > 0} that $P > 0$. Define $Q \coloneqq P^{-1} > 0$. By Proposition \ref{dualisationineq},
\begin{equation} \label{unknownverdissineq}
    \begin{bmatrix} I & \! \! \! \! 0 \\ \mathcal{A}^T & \! \! \! \! \mathcal{C}^T \end{bmatrix}^T \! \! \begin{bmatrix} Q & \! \! \! \! 0 \\ 0 & \! \! \! \! -Q \end{bmatrix} \! \! \begin{bmatrix} I & \! \! \! \! 0 \\ \mathcal{A}^T & \! \! \! \! \mathcal{C}^T \end{bmatrix} + \begin{bmatrix} 0 & \! \! \! \! I \\ E^T & \! \! \! \! F^T \end{bmatrix}^T \! \! \hat{S} \! \begin{bmatrix} 0 & \! \! \! \! I \\ E^T & \! \! \! \! F^T \end{bmatrix} \geq 0
\end{equation}
holds for all $(A, B, C, D) \in \Sigma_u$. We can rewrite the dual dissipation inequality (\ref{unknownverdissineq}) as
\begin{equation*} \label{unknownverrewrittenineq}
\begin{pmat}[{}]
      I \cr \-
      {\begin{matrix} A^T & \! \! \! \! C^T \\ B^T & \! \! \! \! D^T \end{matrix}} \cr
    \end{pmat}^T
M_u
\begin{pmat}[{}]
      I \cr \-
      {\begin{matrix} A^T & \! \! \! \! C^T \\ B^T & \! \! \! \! D^T \end{matrix}} \cr
    \end{pmat} \geq 0,
\end{equation*}
where
\begin{equation*}
    M_u \coloneqq
    \begin{pmat}[{|}]
    \hat{R} & 0 \cr \-
    0 & {\begin{matrix} -Q & -QK^T \\ -KQ & -KQK^T \end{matrix}} \cr
    \end{pmat}.
\end{equation*}
Because the data $(U_{-}, X, Y_{-})$ are informative for closed-loop dissipativity, we have $\mathcal{Z}_{n+m}(N_u) \subseteq \mathcal{Z}_{n+m}(M_u)$. Since $\hat{N}_u$ has a positive eigenvalue, also $N_u$ has a positive eigenvalue and moreover it can be shown that $N_u \in \boldsymbol{\Pi}_{n+p,n+m}$. Therefore, by Proposition \ref{dissmainproof1}, $M_u-\alpha N_u \geq 0$ for some $\alpha \geq 0$. Define $L \coloneqq KQ$. By using a Schur complement argument, we conclude that $\hat{M}_u-\alpha \hat{N}_u \geq 0$.

Conversely, suppose that there exist matrices $Q>0$ and $L$ and a scalar $\alpha \geq 0$ such that $\hat{M}_u-\alpha \hat{N}_u \geq 0$. Let $K \coloneqq LQ^{-1}$. With a Schur complement argument with respect to the bottom right block, we have that $M_u-\alpha N_u \geq 0$. Let $(A, B, C, D) \in \Sigma_u$. After multiplying this inequality from the right by
\begin{equation} \label{leftrightmultiplicationmatrix}
    \begin{pmat}[{}]
      I \cr \-
      {\begin{matrix} A^T & \! \! \! \! C^T \\ B^T & \! \! \! \! D^T \end{matrix}} \cr
    \end{pmat}
\end{equation}
and from the left by its transpose, we get
\begin{equation*}
    \! \! \begin{pmat}[{}]
      I \cr \-
      {\begin{matrix} A^T & \! \! \! \! C^T \\ B^T & \! \! \! \! D^T \end{matrix}} \cr
    \end{pmat}^T
    \! \! \! \! M_u \! \!
    \begin{pmat}[{}]
      I \cr \-
      {\begin{matrix} A^T & \! \! \! \! C^T \\ B^T & \! \! \! \! D^T \end{matrix}} \cr
    \end{pmat} \! \! \! \geq \!
    \alpha \! \!
    \begin{pmat}[{}]
      I \cr \-
      {\begin{matrix} A^T & \! \! \! \! C^T \\ B^T & \! \! \! \! D^T \end{matrix}} \cr
    \end{pmat}^T
    \! \! \! \! N_u \! \!
    \begin{pmat}[{}]
      I \cr \-
      {\begin{matrix} A^T & \! \! \! \! C^T \\ B^T & \! \! \! \! D^T \end{matrix}} \cr
    \end{pmat} \! \! \! \geq \! 0 .
\end{equation*}
By Proposition \ref{dualisationineq}, (\ref{preunknownverdissineq}) holds with $P \coloneqq Q^{-1}$. Thus, the data $(U_{-}, X, Y_{-})$ are informative for closed-loop dissipativity.

For the second statement, let us assume that the data $(U_{-}, X)$ are informative. This implies that there exist $K$ and $P \geq 0$ such that (\ref{preknownverdissineq}) holds for all $(A, B) \in \Sigma_k$. First, suppose that $P > 0$.
After applying the dualisation result, this time we can rewrite the dual dissipation inequality as
\begin{equation} \label{knownverrewrittenineq}
\begin{pmat}[{}]
    I \cr \-
    {\begin{matrix} A^T & \! \! \! 0 \\ B^T & \! \! \! 0 \end{matrix}} \cr
\end{pmat}^T
M_k
\begin{pmat}[{}]
    I \cr \-
    {\begin{matrix} A^T & \! \! \! 0 \\ B^T & \! \! \! 0 \end{matrix}} \cr
\end{pmat}^T \geq 0,
\end{equation}
where $M_k$ is defined as in (\ref{M_k}) with $H$ given by
\begin{equation*}
    \! \begin{bmatrix} 0 & \! \! 0 \\ 0 & \! \! -C_sQC_s^T \! \! -C_sQK^TD_s^T \! \! -D_sKQC_s^T \! \! -D_sKQK^TD_s^T \end{bmatrix}
\end{equation*}
and $Q = P^{-1}$. The data $(U_{-}, X)$ are informative for closed-loop dissipativity, so (\ref{knownverrewrittenineq}) holds for all $(A, B)$ that satisfy (\ref{noisemodelknownver}), but due to the difference in the structures of (\ref{noisemodelknownver}) and (\ref{knownverrewrittenineq}), we are unable to immediately apply Proposition \ref{dissmainproof1}. To alleviate this problem, we invoke Proposition \ref{PiW}, by noting that
\begin{equation*}
    \begin{bmatrix} A^T \\ B^T \end{bmatrix} \begin{bmatrix} I & 0_{n \times p} \end{bmatrix} = \begin{bmatrix} A^T & 0 \\ B^T & 0 \end{bmatrix}.
\end{equation*}
Because $\Phi_{22} < 0$ and $\begin{bmatrix} X_{-}^T & U_{-}^T \end{bmatrix}^T$ is full row rank, the $(2, 2)$ block of $N_k$, i.e.
\begin{equation*}
    \begin{bmatrix} X_{-} \\ U_{-} \end{bmatrix} \Phi_{22} \begin{bmatrix} X_{-} \\ U_{-} \end{bmatrix}^T \! \! \!,
\end{equation*}
is negative definite and as such we have that $\mathcal{Z}_{n+m}(N_k) \begin{bmatrix} I & 0_{n \times p} \end{bmatrix} = \mathcal{Z}_{n+m}(\bar{N}_k)$, where
\begin{align*}
    \bar{N}_k &\coloneqq
\begin{pmat}[{|}]
    {\begin{matrix} I & 0_{n \times p} \end{matrix}} & 0 \cr \-
    0 & I \cr
\end{pmat}^T
N_k
\begin{pmat}[{|}]
    {\begin{matrix} I & 0_{n \times p} \end{matrix}} & 0 \cr \-
    0 & I \cr
\end{pmat} \\
&=
\begin{pmat}[{|}]
    I & X_{+} \cr
    0 & 0_{p \times T} \cr \-
    0 & {\begin{matrix} -X_{-} \\ -U_{-} \end{matrix}} \cr
\end{pmat}
\Phi_{E^T}
\begin{pmat}[{|}]
    I & X_{+} \cr
    0 & 0_{p \times T} \cr \-
    0 & {\begin{matrix} -X_{-} \\ -U_{-} \end{matrix}} \cr
\end{pmat}^T \! \! \!.
\end{align*}
Consequently, $\mathcal{Z}_{n+m}(\bar{N}_k) \subseteq \mathcal{Z}_{n+m}(M_k)$. Given that $\hat{N}_k$ has a positive eigenvalue, it follows that $\bar{N}_k$ has a positive eigenvalue as well. It can also be shown that $\bar{N}_k \in \boldsymbol{\Pi}_{n+p,n+m}$. By Proposition \ref{dissmainproof1}, there exists a scalar $\alpha \geq 0$ such that $M_k-\alpha \bar{N}_k \geq 0$. In the same fashion as the first half of the proof of the first statement, $\hat{M}_k-\alpha \hat{N}_k \geq 0$ holds. This shows that \ref{thm1bP>0} holds.

Second, suppose that the matrix $P$ that satisfies (\ref{preknownverdissineq}) is not positive definite. Let $\xi \in \ker P$ be nonzero. Because the structures of (\ref{preknownverdissineq}) and (\ref{preunknownverdissineq}) are the same, we can follow the proof of Lemma \ref{Lemma: P > 0} to show that $\mathcal{C}_s \xi = 0$, where $\mathcal{C}_s$ is defined as in Definition \ref{informativitydefinitionknown}. After pre- and postmultiplying (\ref{preknownverdissineq}) by $\begin{bmatrix} \xi^T & 0 \end{bmatrix}$ and its transpose, we obtain that $\xi^T\mathcal{A}^TP\mathcal{A}\xi \leq 0$ with $\mathcal{A}$ defined as in Definition \ref{informativitydefinitionknown}. Because $P \geq 0$, we must have that $P\mathcal{A}\xi = 0$ for all $(A, B) \in \Sigma_k$. Take $(A, B) \in \inte(\Sigma_k)$ and define $A_\epsilon \coloneqq \epsilon v\xi^T$, where $\epsilon > 0$ and $v \in \mathbb{R}^n$ is nonzero. Then, $(A + A_\epsilon, B) \in \Sigma_k$ for sufficiently small $\epsilon$. This implies that $P(A + A_\epsilon + BK)\xi = PA_\epsilon\xi = P\epsilon v\xi^T\xi = 0$. However, $\epsilon$ and $\xi$ are nonzero, so $Pv=0$. Since $v$ was an arbitrary nonzero vector, we conclude that $P = 0$. Then, substituting in (\ref{preknownverdissineq}) results in
\begin{equation*}
    \begin{bmatrix} 0 & \! \! \! I \\ \mathcal{C}_s & \! \! \! F \end{bmatrix}^T \! \! S \! \begin{bmatrix} 0 & \! \! \! I \\ \mathcal{C}_s & \! \! \! F \end{bmatrix} \geq 0.
\end{equation*}
From Proposition \ref{Sylvester}, we know that $\mathcal{C}_s = C_s + D_sK = 0$. Therefore, $\im C_s \subseteq \im D_s$ and
\begin{equation*}
    \begin{bmatrix} I \\ F \end{bmatrix}^T S \begin{bmatrix} I \\ F \end{bmatrix} \geq 0.
\end{equation*}

Now, we prove the other direction. If \ref{thm1bP>0} is true, then we can use the same steps as in the second half of the proof of the first statement to show that the data are informative, but we replace (\ref{leftrightmultiplicationmatrix}) with
\begin{equation*}
\begin{pmat}[{}]
    I \cr \-
    {\begin{matrix} A^T & \! \! \! 0 \\ B^T & \! \! \! 0 \end{matrix}} \cr
\end{pmat}.
\end{equation*}
Otherwise, if \ref{thm1bP=0} is true, then we can choose a matrix $K \in \mathbb{R}^{m \times n}$ such that $C_s + D_sK = 0$ and $P=0$ in order for (\ref{preknownverdissineq}) to hold for all $(A, B) \in \Sigma_k$. This proves the theorem.
\end{proof}

Theorem \ref{maintheorem1} allows us to verify informativity for closed-loop dissipativity by using tools for linear matrix inequalities such as MOSEK \protect\cite{MOSEK} in order to check their feasibility. Furthermore, as we have not specified the matrix $S$, our result can be used with any supply rate that satisfies the inertia condition $\inertia (S) = (p, 0, d)$. We demonstrate this in the next section.

\section{Example} \label{section6}
Consider the following system:
\begin{align*}
    \textbf{x}(t+1) &= A_s\textbf{x}(t) + B_s\textbf{u}(t) + \begin{bmatrix} 0.534 \\ 0.233 \end{bmatrix}\textbf{w}(t), \\
    \textbf{y}(t) &= \begin{bmatrix} 0.573 & \! \! \! -0.462 \end{bmatrix}\textbf{x}(t) + 0.857\textbf{u}(t) + 0.474\textbf{w}(t),
\end{align*}
where the ``true'' state matrices are
\begin{equation*}
    A_s = \begin{bmatrix} -0.292 & 1.551 \\ -0.469 & 0.711 \end{bmatrix}, \ B_s = \begin{bmatrix} -0.066 \\ -0.397 \end{bmatrix}.
\end{equation*}
We want to design a controller that renders the system \emph{state-strictly passive} by state feedback, which can be done with an appropriate choice of the matrix $S$ and Theorem \ref{maintheorem1}.
\begin{definition}
    We call system (\ref{noiselessdiss}) \emph{state-strictly passive}, if there exist an $\epsilon > 0$ and a matrix $P \in \mathbb{S}^n$ with $P > 0$ such that
    \begin{equation*}
        \textbf{x}(t)^T P \textbf{x}(t) + 2\textbf{u}(t)^T\textbf{y}(t) - \epsilon ||\textbf{x}(t)||^2 \geq \textbf{x}(t+1)^T P \textbf{x}(t+1)
    \end{equation*}
    holds for all $t \geq 0$ and all trajectories $(\textbf{u}, \textbf{x}, \textbf{y})$ of (\ref{noiselessdiss}).
\end{definition}
This notion of dissipativity is relevant in the context of Lur'e systems, where absolute stability of the system is guaranteed if the linear part of the system is state-strictly passive (see e.g. \protect\cite{Boyd-1994-LMIs, vandeWouw-2008-outputfeedbackLur'e}). Define a new output $\textbf{z} \coloneqq \begin{bmatrix} \textbf{x}^T & \textbf{y}^T \end{bmatrix}^T$ and note that
\begin{equation*}
    \begin{bmatrix} u \\ y \end{bmatrix}^T \! \! \begin{bmatrix} 0 & \! \! I_m \\ I_m & \! \! 0 \end{bmatrix} \! \! \begin{bmatrix} u \\ y \end{bmatrix} - \epsilon ||x(t)||^2 = \begin{bmatrix} u \\ z \end{bmatrix}^T \! \! \begin{bmatrix} 0 & \! \! 0 & \! \! I_m \\ 0 & \! \! -\epsilon I_n & \! \! 0 \\ I_m & \! \! 0 & \! \! 0 \end{bmatrix} \! \! \begin{bmatrix} u \\ z \end{bmatrix} \!.
\end{equation*}

\noindent Thus, we can study state-strict passivity while still working within the framework of this paper by choosing
\begin{equation} \label{statestrictlypassivesupplyrate}
    S = \begin{bmatrix} 0 & 0 & I_m \\ 0 & -\epsilon I_n & 0 \\ I_m & 0 & 0 \end{bmatrix},
\end{equation}
where $\inertia (S) = (n+m, 0, m)$, $m=p$ and $\epsilon > 0$ is a decision variable.

Suppose that the output matrices are known. The knowledge of the $(A_s, B_s)$ matrices is used only for the purpose of generating $T = 30$ noisy input-state measurements. We assume that the noise samples are bounded in norm at all times, i.e. $|w(t)| \leq 1 \ \forall t$. This implies that $W_{-}$ satisfies (\ref{noisemodeldiss}), where
\begin{equation*}
    \Phi = \begin{bmatrix} TI & 0 \\ 0 & -I \end{bmatrix}.
\end{equation*}
The noise samples were randomly drawn from the uniform distribution in the interval $(0, 1)$. The initial state was randomly drawn from the standard normal distribution. The inputs were randomly drawn from the standard normal distribution multiplied by $20$ and the rank condition on the input-state data was satisfied. After verifying whether $\hat{N}_k$ has at least one positive eigenvalue, we use YALMIP \protect\cite{Lofberg-2004-YALMIP} with MOSEK in MATLAB to solve the LMI presented in Theorem \ref{maintheorem1}. The controller $K = \begin{bmatrix} -0.865 & 1.33 \end{bmatrix}$ makes the system state-strictly passive, i.e. dissipative with respect to the supply rate
\begin{equation*}
    s(w,z) = \begin{bmatrix} w \\ z \end{bmatrix}^T S \begin{bmatrix} w \\ z \end{bmatrix},
\end{equation*}
where $S$ is the same as in (\ref{statestrictlypassivesupplyrate}) with $\epsilon = 0.335$. The matrix $P$ that satisfies the dissipation inequality is
\begin{equation*}
    P = \begin{bmatrix} 1.684 & -1.391 \\ -1.391 & 6.246 \end{bmatrix}.
\end{equation*}
Alternatively, if we assume that we possess no knowledge of the output matrices, we generate the outputs using the same input-state and noise measurements and after verifying that $\hat{N}_u$ has at least one positive eigenvalue, we instead have
\begin{equation*}
    \epsilon = 0.44, \ K \! = \! \begin{bmatrix} -0.836 & \! \! 1.175 \end{bmatrix}, \ P \! = \! \begin{bmatrix} 1.667 & \! \! -1.83 \\ -1.83 & \! \! 6.998 \end{bmatrix}.
\end{equation*}

\section{Conclusions and Discussion} \label{section7}
Throughout this paper, we have considered a linear input-state-output system, where the unknown noise is contained in a given subspace and the matrix of noise samples satisfies a QMI. Next, depending on the presence or absence of prior knowledge on the output system matrices, we have derived necessary and sufficient LMI conditions for noisy data to be informative for closed-loop dissipativity. Additionally, from a finite number of informative data samples, we have found a controller that renders our system dissipative. Finally, we have applied our results with the aim of achieving closed-loop state-strict passivity in a numerical example.

We have assumed to have measurements of the system's state, so a natural extension of our results will involve input-output data only. In future work, similar to \protect\cite{vanWaarde-2024-iobehaviouraldatadrivenstabilisation}, we can adopt both the behavioural and informativity approaches in order to study dissipativity properties and design feedback controllers for linear input-output systems in auto-regressive form. Another interesting research line involves developing specific algorithms for solving the LMIs in Theorem \ref{maintheorem1}.

\bibliographystyle{IEEEtran}
\bibliography{references}

\begin{thebibliography}{10}
\providecommand{\url}[1]{#1}
\csname url@samestyle\endcsname
\providecommand{\newblock}{\relax}
\providecommand{\bibinfo}[2]{#2}
\providecommand{\BIBentrySTDinterwordspacing}{\spaceskip=0pt\relax}
\providecommand{\BIBentryALTinterwordstretchfactor}{4}
\providecommand{\BIBentryALTinterwordspacing}{\spaceskip=\fontdimen2\font plus
\BIBentryALTinterwordstretchfactor\fontdimen3\font minus \fontdimen4\font\relax}
\providecommand{\BIBforeignlanguage}[2]{{%
\expandafter\ifx\csname l@#1\endcsname\relax
\typeout{** WARNING: IEEEtran.bst: No hyphenation pattern has been}%
\typeout{** loaded for the language `#1'. Using the pattern for}%
\typeout{** the default language instead.}%
\else
\language=\csname l@#1\endcsname
\fi
#2}}
\providecommand{\BIBdecl}{\relax}
\BIBdecl

\bibitem{vanWaarde-2023-informativityapproach}
H.~J. van Waarde, J.~Eising, M.~K. Camlibel, and H.~L. Trentelman, ``The informativity approach: To data-driven analysis and control,'' \emph{IEEE Control Systems}, vol.~43, no.~6, pp. 32--66, 2023.

\bibitem{DePersis-2020-datadrivencontrol}
C.~De~Persis and P.~Tesi, ``Formulas for data-driven control: Stabilization, optimality, and robustness,'' \emph{IEEE Transactions on Automatic Control}, vol.~65, no.~3, pp. 909--924, 2020.

\bibitem{Coulson-2019-DeePC}
J.~Coulson, J.~Lygeros, and F.~D{\"o}rfler, ``Data-enabled predictive control: In the shallows of the {DeePC},'' in \emph{2019 18th European Control Conference}, Naples, Italy, 2019, pp. 307--312.

\bibitem{Willems-1972-dissipativitypart1}
J.~C. Willems, ``Dissipative dynamical systems part {I}: General theory,'' \emph{Archive for Rational Mechanics and Analysis}, vol.~45, no.~5, pp. 321--351, 1972.

\bibitem{Willems-1972-dissipativitypart2}
------, ``Dissipative dynamical systems part {II}: Linear systems with quadratic supply rates,'' \emph{Archive for Rational Mechanics and Analysis}, vol.~45, no.~5, pp. 352--393, 1972.

\bibitem{Willems-1998-QDFs}
J.~C. Willems and H.~L. Trentelman, ``On quadratic differential forms,'' \emph{SIAM Journal on Control and Optimization}, vol.~36, no.~5, pp. 1703--1749, 1998.

\bibitem{WillemsTrentelman-2002-dissipativesystemsQDFspart1}
------, ``Synthesis of dissipative systems using quadratic differential forms: Part {I},'' \emph{IEEE Transactions on Automatic Control}, vol.~47, no.~1, pp. 53--69, 2002.

\bibitem{TrentelmanWillems-2002-dissipativesystemsQDFspart2}
H.~L. Trentelman and J.~C. Willems, ``Synthesis of dissipative systems using quadratic differential forms: Part {II},'' \emph{IEEE Transactions on Automatic Control}, vol.~47, no.~1, pp. 70--86, 2002.

\bibitem{Maupong-2017-datadrivendissipativity}
T.~M. Maupong, J.~C. Mayo-Maldonado, and P.~Rapisarda, ``On {Lyapunov} functions and data-driven dissipativity,'' \emph{IFAC-PapersOnLine}, vol.~50, no.~1, pp. 7783--7788, 2017.

\bibitem{Romer-2017-iodissipationinequalities}
A.~Romer, J.~M. Montenbruck, and F.~Allg{\"o}wer, ``Determining dissipation inequalities from input-output samples,'' \emph{IFAC-PapersOnLine}, vol.~50, no.~1, pp. 7789--7794, 2017.

\bibitem{Romer-2019-iodatadrivendissipativity}
A.~Romer, J.~Berberich, J.~K{\"{o}}hler, and F.~Allg{\"{o}}wer, ``One-shot verification of dissipativity properties from input-output data,'' \emph{IEEE Control Systems Letters}, vol.~3, no.~3, pp. 709--714, 2019.

\bibitem{Koch-2022-datadrivendissipativity}
A.~Koch, J.~Berberich, and F.~Allg{\"{o}}wer, ``Provably robust verification of dissipativity properties from data,'' \emph{IEEE Transactions on Automatic Control}, vol.~67, no.~8, pp. 4248--4255, 2022.

\bibitem{vanWaarde-2022-isodatadrivendissipativity}
H.~J. van Waarde, M.~K. Camlibel, P.~Rapisarda, and H.~L. Trentelman, ``Data-driven dissipativity analysis: Application of the matrix {S}-lemma,'' \emph{IEEE Control Systems Magazine}, vol.~42, no.~3, pp. 140--149, 2022.

\bibitem{Berberich-2020-datadrivenrobustcontrol}
J.~Berberich, A.~Koch, C.~W. Scherer, and F.~Allg{\"o}wer, ``Robust data-driven state-feedback design,'' in \emph{2020 American Control Conference}, Denver, CO, USA, 2020, pp. 1532--1538.

\bibitem{Steentjes-2022-iocrosscovariancebounds}
T.~R.~V. Steentjes, M.~Lazar, and P.~M.~J. Van~den Hof, ``On data-driven control: Informativity of noisy input-output data with cross-covariance bounds,'' \emph{IEEE Control Systems Letters}, vol.~6, pp. 2192--2197, 2022.

\bibitem{vanWaarde-2023-QMIs}
H.~J. van Waarde, M.~K. Camlibel, J.~Eising, and H.~L. Trentelman, ``Quadratic matrix inequalities with applications to data-based control,'' \emph{SIAM Journal on Control and Optimization}, vol.~61, no.~4, pp. 2251--2281, 2023.

\bibitem{Dancis-1986-sylvesterlawofinertia3}
J.~Dancis, ``A quantitative formulation of {S}ylvester's law of inertia. {III},'' \emph{Linear Algebra and its Applications}, vol.~80, pp. 141--158, 1986.

\bibitem{MOSEK}
\BIBentryALTinterwordspacing
{MOSEK ApS}, \emph{The {MOSEK} optimization toolbox for {MATLAB} manual. Version 9.0.}, 2019. [Online]. Available: \url{http://docs.mosek.com/9.0/toolbox/index.html}
\BIBentrySTDinterwordspacing

\bibitem{Boyd-1994-LMIs}
S.~P. Boyd, L.~El~Ghaoui, E.~Feron, and V.~Balakrishnan, \emph{Linear Matrix Inequalities in System and Control Theory}, ser. SIAM Studies in Applied Mathematics.\hskip 1em plus 0.5em minus 0.4em\relax Philadelphia: Society for Industrial and Applied Mathematics, 1994, vol.~15.

\bibitem{vandeWouw-2008-outputfeedbackLur'e}
N.~van~de Wouw, A.~Doris, J.~C.~A. de~Bruin, W.~P. M.~H. Heemels, and H.~Nijmeijer, ``Output-feedback control of {Lur'e}-type systems with set-valued nonlinearities: A {Popov}-criterion approach,'' in \emph{2008 American Control Conference}, Seattle, WA, USA, 2008, pp. 2316--2321.

\bibitem{Lofberg-2004-YALMIP}
J.~L{\"{o}}fberg, ``{YALMIP}: A toolbox for modeling and optimization in {MATLAB},'' in \emph{2004 IEEE International Conference on Robotics and Automation}, Taipei, Taiwan, 2004, pp. 284--289.

\bibitem{vanWaarde-2024-iobehaviouraldatadrivenstabilisation}
H.~J. van Waarde, J.~Eising, M.~K. Camlibel, and H.~L. Trentelman, ``A behavioral approach to data-driven control with noisy input–output data,'' \emph{IEEE Transactions on Automatic Control}, vol.~69, no.~2, pp. 813--827, 2024.

\end{thebibliography}

\end{document}